\documentclass[review,onefignum,onetabnum]{siamart171218}

\usepackage{lipsum}
\usepackage{amsfonts}
\usepackage{graphicx}
\usepackage{epstopdf}
\usepackage{algorithmic}
\ifpdf
  \DeclareGraphicsExtensions{.eps,.pdf,.png,.jpg}
\else
  \DeclareGraphicsExtensions{.eps}
\fi

\newsiamremark{remark}{Remark}
\newsiamremark{hypothesis}{Hypothesis}
\crefname{hypothesis}{Hypothesis}{Hypotheses}
\newsiamthm{claim}{Claim}

\headers{ Iterative Proximal-Minimization and Differential Game  Model}{S. Gu, H. Zhang, and X. Zhou}

\title{Towards   Robust Calculation of index-$k$ Saddle Point:
 Iterative Proximal-Minimization and Differential Game  Model\thanks{Submitted to the editors DATE.
\funding{This work was funded by the support of NSFC 11901211 and the Natural Science Foundation of Top Talent of SZTU GDRC202137. 
  Xiang Zhou acknowledges the support of Hong Kong RGC GRF grants 11307319, 11308121, 11318522 and  NSFC/RGC Joint Research Scheme  (CityU 9054033).}}}

\author{Shuting Gu\thanks{College of Big Data and Internet, Shenzhen Technology University, Shenzhen 518118, P.R. China 
  (\email{gushuting@sztu.edu.cn}).}
  \and Hao Zhang\thanks{School of Data Science, City University of Hong Kong,
Hong Kong SAR.}
  \and Xiang Zhou\footnotemark[3]\thanks{School of Data Science and Department of Mathematics,
City University of Hong Kong, Hong Kong SAR (\email{xizhou@cityu.edu.hk}). } 
}

\usepackage{amsopn}

\makeatletter
\newcommand*{\addFileDependency}[1]{
  \typeout{(#1)}
  \@addtofilelist{#1}
  \IfFileExists{#1}{}{\typeout{No file #1.}}
}
\makeatother

\ifpdf
\hypersetup{
  pdftitle={Toward Robust Calculation of index-$k$ Saddle Point:
 Iterative Proximal-Minimization and Differential Game  Model},
  pdfauthor={S. Gu, H. Zhang, and X. Zhou}
}
\fi

\usepackage[utf8]{inputenc}
\usepackage[T1]{fontenc}
\usepackage{mathptmx}
\usepackage{etoolbox}

\usepackage{amsfonts,amssymb}
 \usepackage{bm}
\usepackage{graphicx}
\usepackage{color}

\usepackage{hyperref}
\usepackage{multirow}
\usepackage{pifont}

\usepackage{enumitem}

\usepackage{graphicx,epstopdf} 
\usepackage[caption=false]{subfig} 

\newtheorem{thm}{Theorem}
\newtheorem{pro}[thm]{Proposition}
\newtheorem{dfn}{Definition}

\newtheorem{rem}{Remark}
\newtheorem{asm}{Assumption}

\newcommand{\wt}[1]{\widetilde{#1}}

\newcommand{\inpd}[2]{\left\langle #1, #2 \right\rangle}

 \def\Lip{\operatorname{Lip}}

\def\argmin{\operatorname{argmin}}

\newcommand{\bv}[1]{\mathbf{#1}}

\begin{document}
	\nolinenumbers 

\maketitle

\begin{abstract}
 Saddle point with a given Morse index on a potential energy surface
is an important object related to energy landscape in physics and chemistry. Efficient numerical methods based on iterative minimization formulation have been proposed in the forms of  the sequence of minimization subproblems or  the continuous dynamics. We here  
present a differential game  interpretation of this formulation and theoretically investigate  the Nash equilibrium of the proposed  game and the original saddle point on potential energy surface. To  define this differential game, a new  proximal function growing faster than quadratic  is introduced to the cost function in the game and   a   robust Iterative Proximal-Minimization algorithm  (IPM) is then derived to compute the saddle points. 
We prove that the Nash equilibrium of the game is exactly the saddle point in concern and show that the new  algorithm is more robust 
than the previous iterative minimization algorithm without proximity, while    the same convergence rate and the computational cost still hold. A two dimensional  problem and the Cahn-Hillard  problem are tested to demonstrate this numerical advantage.
\end{abstract}

\begin{keywords}
  Iterative Proximal-Minimization, saddle point, transition state, differential game model
\end{keywords}

\begin{AMS}
  68Q25, 68R10, 68U05
\end{AMS}

	\section{Introduction}
	 Saddle points with important physical meaning   have been of broad interest in physics, chemistry, biology and material sciences\cite{Energylanscapes}.
 In computational chemistry, besides multiple local minimum points,  one of the most important objects on the potential energy surface
 is the transition state.  Such transition states are the bottlenecks on the most probable transition paths  between different local wells\cite{cerjan1981,TPSChandler2002,String2002}. In general, the transition state can be 
 described as a special type of the saddle point with Morse index 1, which is defined as the critical point with only one unstable direction.

 In recent years, a large number of numerical methods 
 have been proposed \cite{ String2002, NEB1998,   Dimer1999, GAD2011, IMF2015,DuSIAM2012,npjLeiZhang2016} and developed
 \cite{String2007, Ren2013,SimGAD2018, IMA2015,zhang2016optimization} to efficiently compute these index-1 saddle points.
 Most of them \cite{GAD2011,IMF2015,yin2019high} can be   generalized to the case with the Morse index $k\ge 1$.  In addition, for applications to   multiple unstable solutions of certain nonlinear partial differential equations, there are   computational methods   based on mountain pass and local min-max method  \cite{MPA1993,ZHOUJXSISC2001,ZHOUJXSISC2012} which however do not aim for specific index.

 One important class of algorithms for  saddle points with   a given index 
 is  based on the idea of using the min-mode direction\cite{Crippen1971,Dimer1999, GAD2011,DuSIAM2012, IMF2015, zhang2016optimization,SamantaGAD}.
 The work of  the iterative minimization formulation (IMF) \cite{IMF2015} builds a rigorous mathematical model
 for this min-mode idea and then    a series of IMF-based algorithms  have been developed\cite{LOR2013,IMA2015,ProjIMF}, generalized\cite{MsGAD2017,SimGAD2018} and analyzed \cite{Ortner2016dimercycling}. Although the IMF enjoys the quadratic convergence rate\cite{IMF2015}, 
 the convergence  is only   local and   is subject to the quality of the initial conditions. 
However, in practice, this problem can be much alleviated by using  adaptive inexact solver\cite{GAD2011,IMA2015} for the sub-problems in the IMF.  In most cases, such a trade-off between   efficiency and   robustness works well for applications of these algorithms, but this demands a careful tuning of parameters and a brute-force randomized strategy. On the other side, the continuous model of the IMF as the one-step approximation to the subproblems, which is called gentlest ascent dynamics (GAD) \cite{GAD2011}, is empirically found to be more robust \cite{IMA2015,GAD-DFT2015} than the vanilla IMF,  but GAD only has a linear convergence speed. These current research results thus require a further exploration  of  the underlying  reasons of the numerical divergence  issue. We believe this issue   at least partially comes from the lack of convexity in the sub-problem of optimizing an auxiliary function in the IMF. 
 To investigate this practical challenge, we take a new viewpoint of game theory and its connection to the saddle point.

The IMF defines  an iterative scheme for both a position variable and an orientation vector, whose fixed point, if converged,  is an  index-1 saddle point.  At this fixed point,   the position variable
and the orientation variable minimize  their own objective functions, respectively. This form is very much like a differential game model \cite{2018mechanics}, where the central notion is the Nash equilibrium in game theory \cite{nash1950}. Our main work here is to explicitly investigate such connections between the existing IMF algorithms and differential game models. We shall see that  for a differential game well defined compatible with the IMF, one needs to improve the existing auxiliary function used in the original IMF. 
Our main technique  is to introduce a proximal function as a penalty to the existing auxiliary function to ensure its strict convexity. We show that one has to 
choose the  proximal function growing faster than quadratic function, in contrast to the squared $L_2$ norm  in classic proximal point algorithms. 
With this new proximal technique, we construct a new method, called Iterative Proximal Minimization (IPM) method, as an important improvement of the existing IMF-based algorithms. In theory, we contribute to the proof of  the   equivalence between the following three objects:  saddle point of the potential function, the fixed point of the new iterative scheme, and the Nash equilibrium of the differential game.

The new strategy of using a proper proximal function as a penalty in  this proposed  iterative proximal minimization method is very easy to implement without any extra computational burden than the existing methods. The quadratic convergence rate still preserves.  Most importantly, 
the new iterative proximal method 
 can significantly enhance the robustness  since 
  each subproblem has a well-defined minimizer.
  Extensive numerical  experiences  show that 
when the initial guess is far from the true saddle point,
the subproblem of minimizing the auxiliary function in the original IMF is better to be solved with  a small inner iteration number to maintain the scheme convergent; but for the new method proposed here, the convergence to the saddle point is easier to achieve  regardless how  accurately the subproblem is solved. 
The reason of the improvement in the algorithmic robustness is the convexification of the auxiliary function in the new method when we aim to build the differential game model in a  rational way. We emphasize that due to the special feature of our saddle point iteration method, the penalty function in the auxiliary function can not be quadratic.

To bridge  between the  differential game and the
IMF is more than an academic exposition.  It has practical consequence. 
Generally speaking,  to establish the link from game theory and Nash equilibrium of various optimization problems is quite beneficial both in theory and in algorithmic development\cite{nash1950}. 
The insights from the point of view  of the game theory and Nash equilibrium also have motivated the development of numerical algorithms\cite{gemp2020,DGM2019, omidshafiei2020navigating}.
This work here also contributes to the literature by establishing the connection between saddle point problem in computational physics and the game theory for the first time. As we mentioned before, the new iterative proximal minimization method does not only  allow each sub-optimization problem in the iterative minimization formulation to be well-defined globally, but also leads to the improvement of the robustness of the previous algorithms.

The paper is organized as follows. Section \ref{background} is about the concept of the saddle point and the Morse index,  and the review of the IMF. In Section \ref{Game_theory}, the interpretation of the IMF is given by the game theory.
In Section \ref{IPM}, we first propose the Iterative Proximal Minimization scheme, and prove the equivalence of the Nash equilibrium and the saddle point by introducing the game model, then we present the Iterative Proximal Minimization Algorithm for this new method.
In   Section \ref{Num_ex}, we test two numerical examples: the saddle points of the two dimensional toy model and the   Ginzburg-Landau free energy. Section \ref{con} is the conclusion.

\section{Background}\label{background}
	\subsection{Saddle point and Morse index}
	The saddle point $x^*$ of a potential function $V(x)$ is a critical point at which the partial derivatives of a function $V(x)$ are zero but is not an extremum. The Morse index of a critical point of a smooth function $V(x)$ on a manifold is the negative inertia index  of the Hessian matrix of the function $V(x)$.  
 
{\bf Notations:}
	 \begin{enumerate}
	     \item $\lambda_i(x)\in\mathbb{R}$ for any $x\in\mathbb{R}^d$ is the  $i$-th (in the ascending order) eigenvalue of the Hessian matrix $H(x) = \nabla^2 V(x)$;
	     \item $\bv v_i(x)$ is the eigenvector corresponding to $i$-th eigenvalue $\lambda_i(x)$ of the Hessian matrix $H(x)$;
	     \item $\mathcal{S}_1$ is the collections of all  index-1 saddle points $x^*$ of V, defined by
	     $$
	     \nabla V(x^*) = \mathbf{0}, \quad \text{ and } \lambda_1(x^*)< 0 <\lambda_2(x^*)\cdots\leq \lambda_d(x^*);
	     $$
	     \item $\Omega_{1}$ is a subset of $\mathbb{R}^d$ called index-1 region, defined as
	     $$
	     \Omega_{1}:=\{x\in \mathbb{R}^d:  \lambda_1(x)<0<\lambda_2(x)\};
	     $$
	     \item $\lambda_{\min},\lambda_{\max}:\mathbb{R}^{d\times d}\rightarrow\mathbb{R}$ is the smallest and the largest eigenvalue of a (symmetric) matrix respectively, so by the Courant-Fisher theorem, 
        \begin{equation}\label{eig_mm}
            \lambda_{\min}(A) = \min_{\|z\|=1}z^\top A z,\qquad \lambda_{\max}(A) = \max_{\|z\|=1}z^\top A z;
        \end{equation}
        \item $\mathcal{B}_\epsilon (x) := \{y\in \mathbb{R}^d: \|x-y\| \le \epsilon\}.$
	 \end{enumerate}

	\subsection{Review of  Iterative Minimization Formulation}

     The iterative minimization algorithm to search saddle point of index-1     \cite{IMA2015} is given by
    \begin{equation}\label{eq:imf}
    \left\{\begin{array}{l}
    \bv u^{k}=\argmin_{\|\tilde{\bv u}\|=1} \tilde{\bv u}^{\top} H\left(x^{(k)}\right) \tilde{\bv u}, \\
    x^{(k+1)}=\argmin_{y\in\mathcal{U}(x^k)}W(y;x^{(k)},\bv u^{(k)}),
    \end{array}\right.
    \end{equation}
    where $\mathcal{U}(x)\subset\mathbb{R}^d$ is a local neighborhood of $x$, $H(x) = \nabla^2 V(x)$ and the auxiliary function 
    \begin{equation} \label{W}
            W(y;x, \bv u) = (1-\alpha)V(x) + \alpha V\left(y- \bv u  \bv u^\top(y-x)\right)-\beta V(x +  \bv u  \bv u^\top(y-x)).
    \end{equation}
The constant  parameters $\alpha$ and $\beta$ satisfies $\alpha+\beta>1$.    
    The fixed point $(x^*,\bv u^*)$ of the above iterative scheme should satisfy
    \begin{equation} \label{IMF}
    \left\{\begin{array}{l}
    \bv u^*=\argmin_{\|\tilde{\bv u}\|=1} \tilde{\bv u}^{\top} H\left(x^{*}\right) \tilde{\bv u} =\bv v_1(x^*),\\
    x^{*}=\argmin_{y\in\mathcal{U}(x^*)}W(y;x^*,\bv u^*).
    \end{array}\right.
    \end{equation}
     In addition, suppose that 
    $
    H(x^*) =\bv v\Lambda \bv v^\top,
    $
    where $\Lambda = \operatorname{diag}(\lambda_i)$. Then, a necessary and sufficient condition for
    \[ 
    x^{*}=\underset{y\in\mathcal{U}(x*)}{\argmin}~ W(y;x^*,\bv v_1(x^*))
    \]
    is that 
    \[
    \lambda_{1}(x^*)<0<\lambda_{2}(x^*)<\cdots<\lambda_{d}(x^*).
    \]
    The  formal statement  is quoted below from the paper of the IMF \cite{IMF2015}.
    
    \begin{thm}\label{th1}
    Assume that $V(x)\in\mathcal{C}^3(\mathbb{R}^d;\mathbb{R})$. For each $x$, let $\bv v_1(x)$ be the normalized eigenvector corresponding to the smallest eigenvalue of the Hessian matrix $H(x)=\nabla^2 V(x)$, i.e.
    $$
    \bv v_1(x) = \argmin_{u\in\mathbb{R}^d,\| \bv u\|=1}\bv u^\top H(x) \bv u.
    $$
    Given $\alpha,\beta\in\mathbb{R}$ satisfying $\alpha+\beta>1$, we define the following function of variable $y$,
    \begin{equation}
        W(y;x ) = (1-\alpha)V(x) + \alpha V\left(y-\bv v_1(x)\bv v_1(x)^\top(y-x)\right)-\beta V(x+\bv v_1(x)\bv v_1(x)^\top(y-x))
    \end{equation}
    Suppose that $x^*$ is an index-1 saddle point of the function $V(x)$, i.e,
    $$
    \lambda_{1}<0<\lambda_{2}<\cdots<\lambda_{d}.
    $$
    Then the following statements are true
    \begin{enumerate}
        \item $x^*$ is a local minimizer of $W(y;x^*,\alpha,\beta)$;
        \item There exists a neighborhood $\mathcal{U}(x^*)$ of $x^*$ such that for any $x\in\mathcal{U}(x^*), W(y;x,\alpha,\beta)$ is strictly convex in $y\in\mathcal{U}$ and thus has an unique minimum in $\mathcal{U}(x^*)$;
        \item Define the mapping $\Psi:x\in\mathcal{U}\rightarrow\Psi(x)\in\mathcal{U}$ where $\Psi(x)$ is the unique local minimizer of $W$ in $\mathcal{U}$ for any $x\in\mathcal{U}$. Further assume that $\mathcal{U}$ contains no other stationary point of $V$ expect $x^*$. Then the mapping has only one fixed point $x^*$.
    \end{enumerate}
    \end{thm}

\section{Game Theory interpretation  of IMF}\label{Game_theory}

The Iterative Minimization Formulation is an iterative scheme defined in \eqref{eq:imf}. Various algorithms have been developed based on this formulation\cite{IMA2015,convex_IMF,ProjIMF}.  
Game theory studies the mathematical models of strategic interactions between players, where each player takes an action and receives a utility (or pays a cost) as a function of the actions taken by all agents in the game. 
 Here we are interested in explicitly building a game theory model so that the fixed-point in the iterative minimization scheme \eqref{IMF} in fact is the optimal action taken by the participants in the game.

We start with the definition of a game. A {\it game} of $n$-players is denoted  as $G_n = (P,A,C)$, where $P = \{1,\cdots,n\}$ is the set of $n$ players, $A = A_1\times\cdots\times A_n$ is the action space of these players,  and $C = \{C_i(a_1,\ldots,a_n):A\rightarrow\mathbb{R} \, ,  i\in P\}$ is the set of utility(or cost) functions. Each player $i$ can choose an action from its  own action space with the target to maximize its utility  (or minimize it cost) $C_i$. 

A key  concept in game theory is the {\it Nash equilibrium}\cite{nash1950}, which specifies an action profile under which no player can improve its own utility (or reduce its cost) by changing its own action, while all other players fix their actions. A Nash equilibrium can be a pure action profile or a mixed one, corresponding to specific action of each player or a probability distribution over the action space, respectively.
Here we are  interested in the pure Nash equilibrium and the setting that all players attempt to minimize their costs.
\begin{dfn}[Pure Nash equilibrium]
For an $n$-player game $G_n = \left(P,A,C\right)$, $P = \{1,2,...,n\}$, $A = A_1\times\cdots\times A_n$, $C = \{C_i:A\rightarrow\mathbb{R}, i\in P\}$ denoting 
the set of players, action space and cost functions, respectively, an action profile $a^* = (a^*_1,...,a^*_n)$ is said to be a pure Nash equilibrium if and only if
\[ C_k(a^*_1,...,a^*_k,...,a^*_n)\leq C_k(a^*_1,...,a^*_{k-1},a_k,a^*_{k+1},...,a^*_n), \quad  \forall a_k\in A_k,
\forall k\in P.
\]
\end{dfn}

The auxiliary function $W$ in the IMF \eqref{eq:imf} depends on $y$ and $x$,  where $x$ is the parameter.
We introduce a new player whose  action is $y$ in addition to the two players with action $(x,\bv u)$ and propose a simple penalty cost $\|x-y\|^2$ to enforce the synchronization between the players $x$ and $y$. Then the condition in equation \eqref{IMF} is equivalent to
\begin{equation} \label{eq:imf_fp}
    \left\{\begin{array}{l}
    \bv u^*=\argmin_{\|\tilde{\bv u}\|=1} \tilde{\bv u}^{\top} H\left(x^{*}\right) \tilde{\bv u}, \\
    y^{*}=\argmin_{y\in\mathcal{U}(x^*)}W(y;x^*,\bv u^*),\\
    x^* = \argmin_{y\in\mathbb{R}^d}\frac{1}{2}\|x-y^*\|^2.\\
    \end{array}\right.
\end{equation}

We  now can see that each of $x^*, y^*$ and $\bv u^*$ minimizes its own function of $(x,y,\bv u)$ and
this condition \eqref{eq:imf_fp}  naturally motivates us to understand the fixed points of the iterative scheme as the Nash equilibrium of a game. However, the local neighbourhood $\mathcal{U}(x)$, which is virtually the feasible set for $y$ to make sure the minimization of $W$ is well defined, depends on the action of the player $x$. 
 In practice, the restriction of this local neighbour $\mathcal{U}(x)$ is resolved\cite{GAD2011,IMA2015} by using the parameter $x$ as  the initial guess for minimizing over $y$.
 But in the classic game theory setup,  the action space of each player $A_i$ should be independent of actions taken by other players\cite{osborne1994}. Therefore, we can not  directly formulate a game whose Nash equilibrium is characterized by equation \eqref{eq:imf_fp}.

To remove this   ambiguity of  local constraint from 
$\mathcal{U}(x)$, we aim to penalize the player with action $y$ when $y\notin\mathcal{U}(x)$ by a modified function of $W$, while keeping the minimizers $y^*$ unchanged:
\begin{equation}\label{eq:p}
\displaystyle
\argmin_{y\in\mathcal{U}(x^*)}W(y;x^*,\bv v_1(x^*)) = \argmin_{y\in\mathbb{R}^d}\widetilde{W}_\rho(y;x^*,\bv v_1(x^*)),
\end{equation}
where $\widetilde{W}_\rho$  is a  new penalized cost function of $W$. The choice of this function is crucial and We will present  the detailed ideas and theories  in the next section.

\section{Iterative Proximal-Minimization and Differential Game  Model}\label{IPM}
In this section, we extend the standard IMF by adding a penalty function to $W$ so that equation \eqref{eq:p} holds and the penalized function $\widetilde{W}_\rho$ is continuous and differentiable. This allows us to formulate a game of multi-players and show that the Nash equilibrium of this game coincides with the saddle point of the potential function $V$. 

\subsection{Iterative Proximal Minimization}

We propose the following modified IMF with proximal penalty and call it Iterative Proximal Minimization (``IPM'' in short):
\begin{equation} \label{modified_IMF}
    \left\{\begin{array}{l}
    \bv u^{k}=\argmin_{\|\tilde{\bv u}\|=1} \tilde{\bv u}^{\top} H\left(x^{k}\right) \tilde{\bv u} =\bv v_1(x^k),\\
    x^{k+1}=\argmin_{y\in\mathbb{R}^d}\widetilde{W}_\rho(y;x^k,\bv u^k),
    \end{array}\right.
    \end{equation}
where
    \begin{equation}
    \label{WK}
    \begin{split}
        \widetilde{W}_\rho(y;x, \bv {u}) &= {W}(y;x, \bv {u}) + \rho\  d(x,y)\\
        &=(1-\alpha)V(x) + \alpha V\left(y-\bv u \bv u^\top(y-x)\right)-\beta V(x + \bv u \bv u^\top(y-x)) + \rho\  d(x,y)
    \end{split}
    \end{equation}
    with a positive constant $\rho>0$. Here $d(x,y)$ is a  function on $\mathbb{R}^d\times \mathbb{R}^d$ satisfying the following assumptions:

\begin{asm} \label{asmd}
~
    \begin{enumerate}[label=(\alph*), ref=\ref{asmd}\alph*]
       \item \label{asmd1} For any $x$, $d(x,y)$ is   convex and $C^2$ in $y$, and 
       \[
       \nabla^2_y d(x,y) = \bv 0 ~ \text{ if and only if }~y=x;
       \]
       \item \label{asmd2} $ d(x,y)\geq 0$ for all $x,y$;  and $   d(x,y)=0 $ if and only if $x=y$; 
    \item \label{asmd3}
    For  any constant $\epsilon>0$, there exists a  positive constant  $ \bar{\lambda}_\epsilon>0$, such that  
\[\inf_{\|x-y\| \ge \epsilon}\lambda_{\min}(\nabla^2_y d(x,y))\geq\bar{\lambda}_\epsilon.\]
    \end{enumerate}
   \end{asm}
\begin{rem}
The first condition says that 
$d(x,y)$ is convex but not  strongly convex in $y$.
So  the quadratic function 
$\|x-y\|^2$ does not satisfy this first condition.
The second condition implies that 
$\nabla_y d(x,y)= \bv 0$ at $y=x$.
The third     condition   implies 
the strong convexity in $y$ outside any ball with center at $x$.
The  example  of quartic $d(x,y) = \|x-y\|^4 = \sum_{i=1}^d(x_i-y_i)^4$   satisfies  all conditions in  Assumption \ref{asmd}. 
\end{rem}

        Assumption \ref{asmd1} and Assumption \ref{asmd2} above guarantee that the introduction of $d$ to  $\widetilde{W}_\rho$ in \eqref{modified_IMF} will not change the fixed point of the original IMF. 
   Assumption \ref{asmd3} is for convexification of    $\widetilde{W}_\rho(y;x,\bv v_1(x))$ for $y\in\mathbb{R}^d$ so that 
    $
    \min_{y\in\mathbb{R}^d}\widetilde{W}_\rho(y;x,\bv v_1(x))$
is a strictly convex  problem  with a unique solution.
  
    \subsection{ Differential Game Model   and Nash  Equilibrium  }
    With the introduction of the new $W$ function, we formulate a differential game with specific set of players, action space, and cost functions and 
    can prove the equivalence between Nash equilibrium  and the index-1 saddle point. 
    
    Consider the following game ${G}_3$ played by three players indexed by $\{-1,0,1\}$,
    \begin{table}
    \begin{center}
	\begin{tabular}{|c|c|c|}
		\hline
		Player&Action Variable  &Cost function\\
		\hline
	    ``-1'' & $y\in\mathbb{R}^d$ & $\widetilde{W}_\rho(y;x,\bv u)$\\
	    ``0'' & $x\in\mathbb{R}^d$ & $\frac{1}{2}\|x-y\|^2$\\
	     ``1'' & $\bv u\in\mathbb{S}^{d-1}$ & $\bv u^\top H(x)\bv u$\\
		\hline	
	\end{tabular}
    \end{center}
    \caption{The definition of the 3-player game $G_3$}
    \label{G3}
     \end{table}

    where $\mathbb{S}^{d-1}$ is the unit $L_2$ sphere in $\mathbb{R}^d$.
    Furthermore, we assume the following statements hold for the potential function $V$:
    \begin{asm} \label{asm}
    ~
        \begin{enumerate}[label=(\alph*), ref=\ref{asm}\alph*]
            \item \label{asm1} $V\in \mathcal{C}^3(\mathbb{R}^d)$ and Lipschitz continuous with the Lipschitz constant $\Lip(V)$;
            \item \label{asm2} $V$ has a non-empty and finite set of index-1 saddle points, denoted by $\mathcal{S}_1$;
            \item \label{asm3}  $\nabla^2 V(x)$ is bounded uniformly. That is, there exist two constants $\bar{\lambda}_L,\bar{\lambda}_U$ such that $\bar{\lambda}_L \leq \lambda_{\min}(\nabla^2 V(x))\leq \lambda_{\max}(\nabla^2 V(x)) \leq \bar{\lambda}_U$ for all $x\in \mathbb{R}^d$;
             \item \label{asm4} All stationary points of the potential function $V$ are non-degenerate. That is, $\forall x\in\mathbb{R}^d, i\in\{1,\cdots,d\}$, we have $\lambda_i(x)\neq0$ and $\lambda_i(x)\neq\lambda_j(x)$ for all $i\neq j$.
        \end{enumerate}
    \end{asm}

We start with the iterative proximal minimization \eqref{modified_IMF}. 
Firstly, define a mapping $\Phi_\rho:\mathbb{R}^d\rightarrow 2^{\mathbb{R}^d}$ as the set of minimizers of $\widetilde{W}_\rho(y;x,\bv v_1(x))$ for a given $x\in\mathbb{R}^d$.
\begin{dfn}\label{phi}
Let $\Phi_\rho:\mathbb{R}^d\rightarrow 2^{\mathbb{R}^d}$ be the (set-valued) mapping defined as
\begin{equation}\label{PHI}
x\mapsto   \Phi_\rho(x):=\argmin_{y\in\mathbb{R}^d}\widetilde{W}_\rho(y;x,\bv v_1(x))  
\end{equation}
If the optimization problem has no optimal solution, $\Phi_\rho(x)$ is defined  as an empty set. 
If the optimal solution is not unique, then $\Phi_\rho(x)$ is defined  as the collection of all optimal solutions.
\end{dfn}

We have the following   characterization of  fixed points of  $\Phi_\rho$   and the set of index-1 saddle points of potential function $V$.

\begin{thm}\label{lem:fs} Let   Assumption \ref{asmd} and  \ref{asm} hold, then we have that
 \begin{enumerate}
    \item[(1)]  
    There exists a positive constant $\bar{\rho}$ depending on bounds of  $\nabla^2V$, $\alpha$ and $\beta$ only,
    such that  
    \[
    \Phi_\rho(x^*)= \{x^*\}, \quad \forall x^*\in \mathcal{S}_1, ~
\rho\geq \bar{\rho}.
    \]
    \item[(2)] 
    If $\Phi_\rho(x^*)=\{x^*\}$ with  $\rho>0$, then $x^*\in \mathcal{S}_1$.
    \item[(3)]
   If   $d(x,y)=d(x-y)$, and the iteration sequence $\{x^{(k)}\}$ generated from  $x^{(k+1)} = \Phi_\rho(x^{(k)})$ converges, then the convergence rate is   exactly quadratic.
\end{enumerate}
\end{thm}
We remark  that  the choice  of the penalty factor $\bar{\rho}$ does not depend on the 
specific choice of saddle point $x^*$: 
the equivalence statement here holds  for all index-1 saddle points  in $\mathcal{S}_1$. 

The proof of Theorem \ref{lem:fs} needs the following two propositions. 
\begin{pro}\label{pro:lip}
Let $V$ be a  Lipschitz continuous function   with the Lipschitz constant $\Lip(V)$, then for any point $x\in \mathbb{R}^d$, any unit vector $\bv v\in \mathbb{S}^{d-1}$, and $\alpha+\beta>1$,
the function  \[ y\mapsto 
W(y;x,\bv v) = (1-\alpha)V(x) + \alpha V\left(y-\bv v \bv v^\top(y-x)\right)-\beta V(x + \bv v \bv v^\top(y-x))
\]
is also  Lipschitz with the Lipschitz constant
$\operatorname{Lip}(W) = (\alpha+\beta)\Lip(V)$ and $\operatorname{Lip}(W)$ is independent of $x$ and $\bv v$.
\end{pro}

 \begin{proof}
 For each $x\in\mathbb{R}^d$ and any $y_1,y_2\in\mathbb{R}^d$, we have
\begin{align*}
\begin{split}
     \| W(y_1;x,\bv v)  -W(y_1;x,\bv v)\|  
 \leq  &  \alpha\|V\left(y_1-\bv v \bv v^\top(y_1-x)\right)-V\left(y_2-\bv v \bv v^\top(y_2-x)\right)\|
    \\
    & +\beta\|V(x + \bv v\bv v^\top(y_1-x)) - V(x + \bv v\bv v^\top(y_2-x))\|
    \\
    \leq  & \alpha\Lip(V)\|\Big[\mathbf{I}-\bv v\bv v^\top\Big](y_1-y_2)\|+\beta\Lip(V)\|\bv v\bv v^\top(y_1-y_2)\|
    \\
    \leq & (\alpha+\beta)\Lip(V)\|y_1-y_2\|. 
\end{split}
\end{align*}
\end{proof}

The second proposition below is about the existence of 
  the unique minimizer in \eqref{PHI} for $x\in\Omega_1$.  
\begin{thm}\label{pro:convex}
Suppose assumption \ref{asm} holds, then
for any compact subset ${\Omega}^\prime_1$ of the index-1 region $\Omega_{1}$ and two constants $\alpha+\beta>1$, there exists a constant $\bar{\rho}$ depending on $\alpha, \beta$, and $\Omega^\prime_1$, such that for all $\rho>\bar{\rho}$, the following optimization problem of $y$,
    \begin{equation}\label{eq:min}
         \begin{split}
      \min_{y\in\mathbb{R}^d}\widetilde{W}_\rho(y;x,\bv v_1(x)) = {W}(y;x,\bv v_1(x)) + \rho\cdot d(x,y)\\
    \end{split}
    \end{equation}
 where $\bv v_1(x)$ is the smallest eigenvector of the Hessian matrix $H(x)=\nabla^2 V(x)$, 
has a unique solution for each $x \in {\Omega}^\prime_1$, i.e., $\Phi_\rho(x)\neq \emptyset$ and $\Phi_\rho(x)$ is a singleton.
\end{thm}

\begin{proof}
To prove our conclusion,  we will claim the optimization  is a strictly convex problem \eqref{eq:min} by showing that  at a sufficiently large $\rho$,  $\widetilde{W}_\rho$ is a strictly convex function of $y$ in $\mathbb{R}^d$ uniformly for  $x \in \Omega_1^\prime$.
    This will be proved by showing the minimal eigenvalues of the Hessian  matrix is positive. 
    
    The Hessian matrix of $\widetilde{W}_\rho$   with respect to $y$ is 
    \begin{equation} \label{eq:HWT}
    \begin{split}
    \widetilde{H}_\rho(y;x)&:=\nabla_y^2  \widetilde{W}_\rho(y; x,\bv v_1(x))=(1-\alpha) H(y) +\alpha\big[\mathbf{I} - \Pi_1(x)\big]H(y-\Pi_1(x)(y-x))\big[\mathbf{I} - \Pi_1(x)\big]\\
    &\qquad- \beta\Pi_1(x) H(x + \Pi_1(x)(y-x))\Pi_1(x) + \rho\nabla^2_y d(x,y)\\
    & =:\mathcal{H}(y;x) + \rho\nabla^2_y d(x,y),
    \end{split}
    \end{equation}    
    where $\mathbf{I}$ is the identical matrix of $\mathbb{R}^{d\times d}$, $\Pi_1(x) := \bv v_1(x)\bv v_1(x)^\top$, and the symmetric matrix
     \begin{align*}
    \begin{split}
        \mathcal{H}(y;x) &:= (1-\alpha)H(y) +\alpha\big[\mathbf{I} - \Pi_1(x)\big]H(y-\Pi_1(x)(y-x))\big[\mathbf{I} - \Pi_1(x)\big]\\
    &\qquad- \beta\Pi_1(x) H(x + \Pi_1(x)(y-x))\Pi_1(x).
    \end{split}
    \end{align*}
    By the inequality $$\lambda_{\min}\left(\widetilde{H}_\rho(y;x)\right)\geq \lambda_{\min}(\mathcal{H}(y;x)) + \rho\lambda_{\min}(\nabla^2_y d(x,y)),
    $$ we focus on $\lambda_{\min}(\mathcal{H}(y;x))$ first.
    
    Given a compact set $\Omega_1^\prime$ in index-1 region, 
    we first fix a point $x \in \Omega_1^\prime\subset \Omega_1$.
Then $\lambda_1(x)<0<\lambda_2(x)$.
Note the eigenvectors of 
\begin{align*}
    \mathcal{H}(x;x) = (1-\alpha)H(x) + \alpha(\mathbf{I}-\Pi_1(x))H(x)(\mathbf{I}-\Pi_1(x)) - \beta\Pi_1(x)H(x)\Pi_1(x)
\end{align*}
coincide with the eigenvectors of the Hessian matrix $H(x)$, because  for $i\neq1$, we have
\begin{align*}
\begin{split}
    \mathcal{H}(x;x)\bv v_i(x) &= (1-\alpha)H(x)\bv v_i(x) +\alpha(\mathbf{I}-\Pi_1(x))\lambda_i(x)\bv v_i(x) + 0\\
    &=(1-\alpha)\lambda_i(x)\bv v_i(x) + \alpha\lambda_i(x)\bv v_i(x)\\
    &=\lambda_i(x)\bv v_i(x)
\end{split}
\end{align*}
and at $i=1$, 
\begin{align*}
\begin{split}
    \mathcal{H}(x;x)\bv v_1(x) & =(1-\alpha)\lambda_1(x)\bv v_1(x) + 0 - \beta\lambda_1(x)\bv v_1(x)\\
    &=(1-\alpha-\beta)\lambda_1(x)\bv v_1(x).
\end{split}
\end{align*}
Therefore,  the eigenvalues of $\mathcal{H}(x;x)$ are given by
\[
\{(1-\alpha-\beta)\lambda_1(x),\lambda_2(x),\cdots,\lambda_d(x)\}
\]
and they are all strictly positive since
$\alpha+\beta>1$ and $x\in \Omega_1$.
Therefore $\mathcal{H}(x;x)$ is positive definite. In addition, since $V\in\mathcal{C}^2(\mathbb{R}^d)$,
 for each $x$ we have that $\lambda_{\min} (\mathcal{H}(y;x))>0$ for all $y$ inside a ball neighbourhood $\mathcal{B}_{\varepsilon_x}(x)$  with radius $\epsilon_x>0$ depending on $x$, by the continuity of $\nabla^2 V$.
 We  can choose this radius continuously depending on $x$ and
pick up the smallest radius  $$\bar{\varepsilon}=\min_{x\in \Omega_1^\prime} \epsilon_x>0$$
  which is strictly positive since $\Omega_1^\prime$ is compact.
  This means that   
 \[
  \inf_{x\in \Omega'_1}\inf_{\|y-x\|\le \bar{\epsilon}} \lambda_{\min}( \mathcal{H}(y;x) ) >0,
    \]
    which implies  for any $\rho\ge 0$, the Hessian matrix $\widetilde{H}_\rho$
    satisfies the same condition 
    \begin{equation}
    \label{eq:620}
      \inf_{x\in \Omega'_1}\inf_{\|y-x\|\le \bar{\epsilon}} \lambda_{\min}( \wt{H}_\rho(y;x) ) >0,
    \end{equation}
      since $\nabla^2_y d(x,y)\succeq \bv 0$ due to {\bf Assumption} \ref{asmd1}.
    
    In order to show  $\widetilde{H}_\rho$  is also positive definite in $y\in \mathbb{R}^d$ for all $x$ in $\Omega_1^\prime$, we need to choose  a sufficiently large penalty factor $\rho$.
   By {\bf Assumption} \ref{asmd3},  there exists a constant  $\bar{\lambda}_{\bar{\varepsilon}}>0$, such that $\lambda_{\min}(\nabla^2_y d(x,y)) \ge  \bar{\lambda}_{\bar{\varepsilon}}$  for any $x,y$ satisfying   $y\notin\mathcal{B}_{\bar{\varepsilon}}(x)$. 
Recall that from {\bf Assumption} \ref{asm3}, the Hessian matrix of potential function $V$ is bounded everywhere. Let $\bar{\lambda} = \max\{|\bar{\lambda}_L|,|\bar{\lambda}_U|\}$, then we have the lower bound of the minimal eigenvalue 
\begin{align*}\label{422}
\begin{split}
\lambda_{\min}\left(\wt{H}_\rho(y;x)\right)
&\geq \lambda_{\min}(\mathcal{H}(y;x)) + \rho\lambda_{\min}(\nabla^2_y d(x,y))\\
&\geq -(|1-\alpha| + |\alpha| + |\beta|) \bar{\lambda} + \rho\lambda_{\bar{\epsilon}}.
\end{split}
\end{align*}
Let  $\rho>\bar{\rho} := {(1+2|\alpha|+|\beta|)\bar{\lambda}}/{\bar{\lambda}_{\bar{\varepsilon}}}>0$, then
\[ \displaystyle  \inf_{x\in \Omega'_1}\inf_{\|y-x\|>\bar{\epsilon}}
~\lambda_{\min}\left(\wt{H}_\rho(y;x)\right)>0.\]

Therefore, we conclude that when $\rho>\bar{\rho}$, 
\[ \inf_{x\in \Omega'_1} \inf_{y\in  \mathbb{R}}~\lambda_{\min}\left(\wt{H}_\rho(y;x)\right)>0.
 \]
That is, $\widetilde{W}_\rho(y;x,\bv v_1(x))$ is strongly  convex in $y$   and the minimization problem in equation 
\eqref{eq:min}  has a unique solution $\Phi_\rho(x)$ for all $x\in \Omega'_1$.
\end{proof}

\begin{rem} $\bar{\rho}$ depends on the uniform bound  of Hessian $\nabla^2 V$, two constants  $\alpha$, $\beta$, and the compact  subset  $\Omega'_1\subseteq \Omega_1  $.
 It is not guaranteed that  $\Phi_\rho(x)$ with $x\in \Omega_1$
   always lies in $\Omega_1$. In addition, 
we can not generalize the conclusion  from $x\in \Omega_1$ to  all $x\in\mathbb{R}^d$ since it is not true at saddle point with index-$k$ when $k > 1$ and the auxiliary function $W$ here is designed for $k=1$.

\end{rem}
\medskip
We   established  the equivalence between the fixed points of the map $\Phi_\rho(\cdot)$ and the index-1 saddle point of potential function $V(\cdot)$.
Then, we are ready to present the   proof of Theorem \ref{lem:fs}.


\begin{proof}[{\bf Proof of Theorem \ref{lem:fs}}]
``Proof of Statement (1)''\\
From Theorem \ref{pro:convex}, we know that for each index-1 saddle point $x_i^*\in\mathcal{S}_1$, there exists a $\bar{\rho}_{i}$ such that $\widetilde{W}_\rho(y;x_i^*,\bv v_1(x_i^*))$ is a convex function of $y\in\mathbb{R}^d$ for all $\rho>\bar{\rho}_i$. Together with assumption \ref{asm2}, we have an uniform $\bar{\rho}:=\max_i \bar{\rho_i}>0$, such that,
if $x^*$ is an index-1 saddle point of potential function $V$, then for any $\rho>\bar{\rho}$, $\widetilde{W}_\rho(y;x^*,\bv v_1(x^*))$ is a convex function of $y\in\mathbb{R}^d$. 
Since we have proved  $\widetilde{W}_\rho(y;x^*,\bv v_1(x^*))$
is strictly convex for all $y\in\mathbb{R}^d$ and $\rho>\bar{\rho}$, we only need to show the first order condition  holds.
Note that
\begin{align}
    \label{1st_eq}
    \nabla_y \widetilde{W}_\rho(y;x, \bv v_1(x)) = &(1-\alpha)\nabla V(y)  + \alpha(\mathbf{I}- \Pi_1)\nabla V(y- \Pi_1(y-x)) \nonumber\\
    & - \beta  \Pi_1 \nabla V(x+ \Pi_1(y-x)) + \rho\nabla_y d(x,y),
\end{align}
where $\Pi_1 =\Pi_1(x)= \bv v_1(x)\bv v_1(x)^\top$,
gives
$\nabla_y \widetilde{W}_\rho(y;x, \bv v_1(x))\vert_{y=x} = \Big[\mathbf{I}-(\alpha+\beta)\Pi_1(x)\Big]\nabla V(x)$ by Assumption \ref{asmd2}.
So, $ \nabla_y \widetilde{W}_\rho(y;x^*, \bv v_1(x^*))\vert_{y=x^*} = \mathbf{0},
$  since $\nabla V(x^*)= \bv 0$. 

    ``Proof of Statement (2) '':\\
    Now we assume $\Phi_\rho(x^*) = \{x^*\}$, i.e., $x^*$ is  the unique minimizer in \eqref{PHI} and we want to show $x^*$ is an index-1 saddle point.
    Then the first order condition $ \nabla_y \widetilde{W}_\rho(y;x^*, \bv v_1(x^*))\vert_{y=x^*} = \mathbf{0}$,  holds and the Hessian matrix $\nabla^2_y\widetilde{W}_\rho(y;x^*,\bv v_1(x^*))|_{y = x^*}$ is positive semi-definite. 
    By the first order condition, we have
    \begin{equation}\label{f_order}
    \begin{split}
    &\nabla \widetilde{W}_\rho(y;x,\bv v_1(x))|_{y = x} =\Big[\mathbf{I}-(\alpha+\beta)\Pi_1(x)\Big]\nabla V(x) = \mathbf{0}.
    \end{split}
    \end{equation}
    Since  $\alpha+\beta>1$, $\nabla V(x)= \bv 0$ holds.
    For the second order condition, by \eqref{eq:HWT}, we have
    \[ \nabla^2 \widetilde{W}_\rho(y;x,\bv v_1(x))|_{y = x}
    = \mathcal{H}(x;x)\succeq \bv 0.
    \]
    
    From the proof of theorem \ref{pro:convex}, we know that the eigenvalues of the Hessian matrix
    \[
    \mathcal{H}(x;x) = \nabla^2 \widetilde{W}_\rho(y;x,\bv v_1(x))|_{y = x}
    \]
    is given by
    \[
    \{(1-\alpha-\beta)\lambda_1(x),\lambda_2(x),\cdots,\lambda_d(x)\}.
    \]
    Then we have that
    \[
    \nabla^2 \widetilde{W}_\rho(y;x,\bv v_1(x))|_{y = x}\succeq \bv 0
    \]
    is equivalent to
    \[
    \lambda_1(x) \le 0\le \lambda_2(x)<\cdots<\lambda_d(x).
    \]
    By assumption \ref{asm4}, we know that $x$ is non-degenerate as $\nabla V(x) = \mathbf{0}$, thus $\lambda_1(x),\lambda_2(x)\neq 0$, which implies
    \[
    \lambda_1(x)<0<\lambda_2(x)<\cdots<\lambda_d(x).
    \]
    Together with $\nabla V(x) = \mathbf{0}$, we conclude that $x = x^*$ is an index-1 saddle point of potential function $V(\cdot)$. 
    
     ``Proof of Statement (3) '':\\
    Now we prove the quadratic convergence rate of the iterative scheme.
     The main idea is very similar with the proof in the reference \cite{IMF2015}.
     The key point is to show  the derivative of the mapping $\Phi_\rho(x)$ vanishes at the saddle point $x^*$.
    For each $x$, $\Phi_\rho(x)$ is a solution of the first order equation \eqref{1st_eq}
    \begin{align}
    \label{1st_eq1}
     (1-\alpha)\nabla V(\Phi_\rho(x)) \! +\! \alpha(\mathbf{I}-\Pi_1(x))\nabla V(\varphi_1(x)) \! - \! \beta \Pi_1(x) \nabla V(\varphi_2(x)) + \rho\nabla_y d(x,\Phi_\rho(x)) = 0,
\end{align}
with
$ 
\varphi_1(x) := \Phi_\rho(x) - \Pi_1(x)(\Phi_\rho(x)-x),  \varphi_2(x) := x + \Pi_1(x)(\Phi_\rho(x)-x).
$
    Taking derivative w.r.t $x$ on both sides of \eqref{1st_eq1} again, we get
     \begin{align}
    \label{2nd_eq}
     & (1-\alpha)H(\Phi_\rho) D_x \Phi_\rho  + \alpha(\mathbf{I}-\Pi_1(x)) H(\varphi_1(x))D_x \varphi_1 - \alpha \bv v_1(x)^\top \nabla V(\varphi_1) J \nonumber \\
     & - \alpha \bv v_1(x) \nabla V(\varphi_1)^\top  J - \beta \Pi_1(x) H(\varphi_2(x)) D_x \varphi_2 - \beta \bv v_1(x)^\top \nabla V(\varphi_2) J \nonumber \\
     & - \beta \bv v_1(x) \nabla V(\varphi_2)^\top  J + \rho D_x (\nabla_y d(x,\Phi_\rho(x))) = 0,
\end{align}
where $J=\frac{\partial \bv v_1}{\partial x}$ and $D_x \Phi_\rho = \frac{\partial \Phi_\rho}{\partial x}$.
Note at $x=x^*$, we have  
$ \Phi_\rho(x^*)  = \varphi_1(x^*) = \varphi_2(x^*) = x^*, \nabla V(x^*)   = \bv 0,  $
and since $d(x,y)=d(x-y)$,
so $\nabla_x \nabla_yd(x,y)=
-\nabla^2 d(x-y)$, which gives 
$ \nabla_x (\nabla_y d(x^*,\Phi_\rho(x^*))) = \nabla_x \nabla_y d(x^*,x^*)=0$.
In addition,
\begin{align*}
    D_x \varphi_1(x^*) & = (\mathbf{I}-\Pi_1(x^*)) D_x \Phi_\rho(x^*) + \Pi_1(x^*), 
 \\
 D_x \varphi_2(x^*)  & = \mathbf{I} - \Pi_1(x^*) + \Pi_1(x^*) D_x \Phi_\rho(x^*).
\end{align*}
So at $x=x^*$,  \eqref{2nd_eq} becomes
\begin{equation}
    \Big( H(x^*) - (\alpha+\beta)\lambda(x^*) \Pi_1(x^*) \Big) D_x \Phi_\rho(x^*) = 0,
\end{equation}
which implies that $D_x \Phi_\rho(x^*) = 0$ if and only if $\alpha+\beta \neq 1$. The second order derivative of $\Phi_\rho(x^*)$  is not trivial 0. This illustrates that the iterative scheme $x\rightarrow \Phi_\rho(x)$ is of quadratic convergence rate. The proof of theorem \ref{lem:fs} is complected.
\end{proof}


Next we can also show  the relation  between fixed points of   $\Phi_\rho$  and the  Nash equilibrium  of the game $G_3$. 
\begin{thm}\label{lem:ns}
For any $\rho>0$, an action profile $(y^*,x^*,\bv u^*)$ is a Nash equilibrium of $G_3$ defined in Table \ref{G3} if and only if  $y^*=x^*, \bv u^*=\bv v_1(x^*)$ and $x^*$ satisfies  $\Phi_\rho(x^*) = \{x^*\}$.
 \end{thm}
  
  \begin{proof}
  The proof is simple by using definitions.
  \end{proof}


 Theorem \ref{lem:fs} and Theorem \ref{lem:ns} together directly lead to
the main result of  Theorem \ref{thm:main}.
    \begin{thm}\label{thm:main}
Suppose that Assumption \ref{asmd} and Assumption \ref{asm}  hold. There exists a positive constant $\bar{\rho}>0$, such that for any sufficiently large penalty factor $\rho>\bar{\rho}$, the following two statements are equivalent:
    \begin{enumerate}
        \item $x^*$ is an index-1 saddle point of $V$;
        \item $(x^*,x^*,\bv v_1(x^*))$ is a strict pure Nash equilibrium of the game $G_3$,
    \end{enumerate}
    where $\bv v_1(x^*)$ is the eigenvector of the Hessian matrix $H(x^*) =\nabla^2 V(x^*)$ corresponding to the smallest eigenvalue $\lambda_1(x^*)$. 
    \end{thm}

 \subsection{Algorithms}
 
 The key improvement in our new method 
 is to add a non-quadratic penalty function $d$ satisfying Assumption \ref{asmd} to the original auxiliary function in the IMF. This iterative proximal minimization method not only offers a well justified 
 game theory model, but also shows the numerical advantage of improving the robustness of the existing algorithms based on the IMF, which will be demonstrated by examples below.
 We point out that the modification of the existing algorithm is extremely simple, and for completeness, we list the main steps in Algorithm \ref{alg:algorithm1}.  
 We comment that in practice the two subproblems of minimization are solved only inexactly in practice, like any existing IMF-based algorithms\cite{IMA2015}.
But  when the minimization  takes only one single gradient step (e.g. $M=1$ in Algorithm \ref{alg:algorithm1}),  $x_{k+1}=x_k - \Delta t_k \nabla \widetilde{W}_\rho(x_k;x_k,\bv u_k)=x_k - \Delta t_k \nabla {W}(x_k;x_k,\bv u_k)$ due to Assumption  \ref{asmd} on function $d$.
We choose the penalty function $d$ as the quartic function
$d(x,y)=|x-y|^4$ in all numerical tests.
 
 \begin{rem}
Our new method could be called iterative penalized minimization scheme, since $\rho d(x,y)$ in $\widetilde{W}_\rho$ is similar to a role of penalty.
However, the main functionality of $d$ is to encourage $x_{k+1}$ close  
to $x_k$, but without affecting the Hessian at $x_k$ by
excluding the common quadratic penalty function.   So we prefer to calling it {\it  iterative proximal minimization}
and $x_{k+1}=\Phi_\rho(x_k)$ could be regarded as a  proximal operator.
\end{rem}

\begin{algorithm}
\caption{Iterative Proximal Minimization Algorithm}
\label{alg:algorithm1}
\begin{algorithmic}
\STATE{{\bf{Input}}: initial guess $x_0$, $\rho>0$, $tol>0$.}
\STATE{ {\bf{Output}}: saddle point $x$.}
\STATE{{\bf{begin}}}
\STATE{Solve the min-mode $ \bv u_0=\argmin \bv u^\top H(x_0) \bv u$;}
\STATE{$k=0$;}
\STATE{ $g_0 = |\nabla V(x_0)|$; \qquad \qquad \qquad \qquad  \qquad \quad // calculate the norm of force }
\WHILE{$g_k > tol$}
\STATE{$k=k+1$;}
\FOR{$i= 1, 2, \cdots, M$}
    \STATE{Calculate  $\widetilde{W}_\rho$ based on \eqref{WK};}
    \STATE{$y_{i+1} = y_i - \Delta t*\nabla \widetilde{W}_\rho$; \qquad\qquad\qquad    // solve gradient flow to update $y_k$}
\ENDFOR
\STATE{	$x_k = y_{M+1}$;}
\STATE{	$\bv u_k  = \argmin{\bv u^\top H(x_k) \bv u}$; \qquad\qquad\qquad  // solve the minimal eigenvector }
\STATE{ $g_k = |\nabla V(x_k)|$. \qquad\qquad\qquad\qquad\quad  // update the force }
\ENDWHILE
\RETURN $x_k$
\end{algorithmic}
\end{algorithm}

\subsection{Generalization to high index saddle point}
 The conclusions and algorithms could be extended to index-$k$ saddle points easily. For any $k\in\{1,\cdots,d-1\}$, we consider the penalized proximal cost function the following form:
\begin{equation}\label{p_w}
    \widetilde{W}_\rho(y;x,\bv v_{1:k}(x)) = (1-\alpha)V(x) + \alpha V(y - \bv v_{1:k}\bv v_{1:k}^\top) - \beta V(x + \bv v_{1:k}\bv v_{1:k}^\top(y-x))
    + \rho d(x,y),
\end{equation}
where $\bv v_{1:k} = (\bv v_1,\cdots, \bv v_k)$ and each $\bv v_i\in\mathbb{S}^{d-1}$ is the i-th eigenvector of $H(x)$ corresponding to the eigenvalue $\lambda_i$ (recall that $\lambda_1<\lambda_2<\cdots<\lambda_d$). 

Then for any $x$ in the index-$k$ region 
\[
\Omega_k = \{x\in\mathbb{R}^d : \lambda_1(x)<\cdots<\lambda_k(x)<0<\cdots<\lambda_d(x)\},
\]
we can extend the mapping $\Phi_\rho$ defined in definition \ref{phi} to the case of index-$k$ saddle points
\[
\Phi_\rho(x) = \argmin_{y\in\mathbb{R}^d}\widetilde{W}_\rho(y;x,\bv v_{1:k}(x)). 
\]
Let $x^*_k$ be the index-$k$ saddle point of $V$  such that
\[
\nabla V(x^*_k) = 0 \text{ and }\lambda_1(x^*_k)<\lambda_2(x^*_k)<\cdots<\lambda_k(x^*_k)<0<\cdots<\lambda_d(x^*_k),
\]
then we can extend the conclusion from Theorem \ref{lem:fs} so that we have
\[
\Phi_\rho(x) = \{x\}\text{ if and only if } x = x^*_k.
\]
The corresponding game $G_{k+2}$ follows,
\begin{center}
	\begin{tabular}{|c|c|c|}
		\hline
		Player&Action&Cost function\\
		\hline
	    ``-1'' & $y\in\mathbb{R}^d$ & $\widetilde{W}_\rho(y;x,\bv u_{1:k})$\\
	    ``0'' & $x\in\mathbb{R}^d$ & $\frac{1}{2}\|x-y\|^2$\\
	     ``1'' & $\bv u_1\in\mathbb{S}^{d-1}$ & $\bv u_1^\top H_1(x)\bv u_1$\\
	     ``2'' & $\bv u_2\in\mathbb{S}^{d-1}$ & $\bv u_2^\top H_2(x)\bv u_2$\\
	     $\cdots$ & $\cdots$ & $\cdots$\\
	     ``$k$'' & $\bv u_k\in\mathbb{S}^{d-1}$ & $\bv u_k^\top H_k(x)\bv u_k$\\
		\hline	
	\end{tabular}
    \end{center}
where $H_i(x) = H(x) - \sum_{j<i}\lambda_j(x)\bv v_j(x)\bv v_j(x)^\top$. Then we can extend  Theorem \ref{thm:main} so that we have
$x = x^*_k$ if and only if  $(x,x,\bv v_{1:k}(x))$ is the Nash equilibrium of the corresponding game $G_{k+2}$.

In addition, Algorithm \ref{alg:algorithm1} could be extended to index-$k$ saddle points easily as well. Compared with the index-1 saddle point, we need to solve the top $k$ eigenvectors of the Hessian matrix $H(x)$ and substitute the cost function with the function in equation \eqref{p_w}, for each step of the iteration in Algorithm \ref{alg:algorithm1}. Yet  we do not intend to pursue this specific numerical issues  about computation of the top $k$ eigen-space in this work. 
 
 \section{Numerical results}\label{Num_ex}
 
In this section, we will illustrate the above new method by a two-dimensional ODE toy model  and a one-dimensional partial differential equation -- the Cahn-Hilliard equation.

 \subsection{A simple  example}
 Consider the following two dimensional potential function
 \begin{equation}\label{energy}
\begin{split}
V(x,y) = 
&3\exp(-x^2 - (y-\frac{1}{3})^2)-3\exp(-x^2 - (y-\frac{5}{3})^2)
     -5\exp(-(x-1)^2 - y^2)
     \\
     &-5\exp(-(x-1)^2 - y^2)
     +\exp(\frac{1}{5}x^4 + \frac{1}{5}(y-\frac{1}{3})^4).
\end{split}
 \end{equation}
 The energy function \eqref{energy} has three local minima approximately at $(1,0), (-1,0),$ and $(0,1.5)$, a maximum at $(0,0.5)$ and three saddle points at $(0.61727,1.10273),(-0.61727,1.10273)$, and $(0, -0.31582)$. 
 
 In this   experiment, we study the convergence properties of the  iterative proximal minimization   scheme when the 
 penalty factor $\rho$ varies.
 The original IMF-based methods\cite{IMA2015,IMF2015} correspond to $\rho=0$  with the same tuning the parameter $M$ (See Algorithm \ref{alg:algorithm1}).
It has been observed before\cite{IMA2015} that to use  small $M$ is usually more robust but slow in convergence, a large $M$ help utilize the theoretic quadratic convergence rate but the scheme then may be quite sensitive to the initial guess and show oscillation\cite{Ortner2016dimercycling} or divergence.

 \begin{figure}[htbp]
	\centering
	\includegraphics[width=0.3\linewidth]{"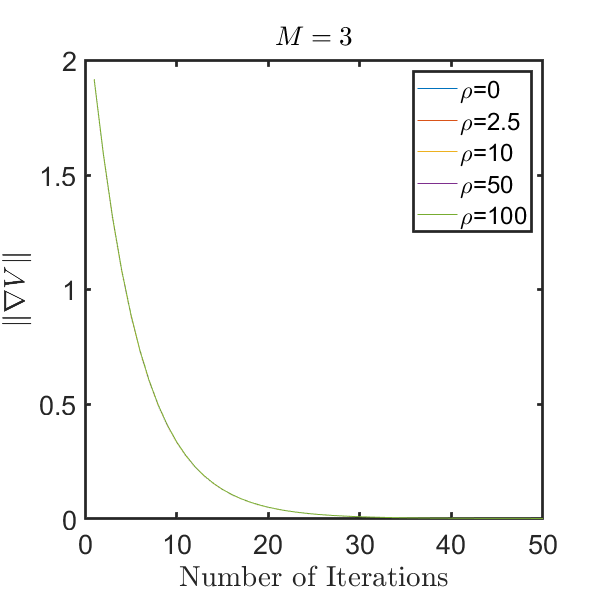"}
	\includegraphics[width=0.3\linewidth]{"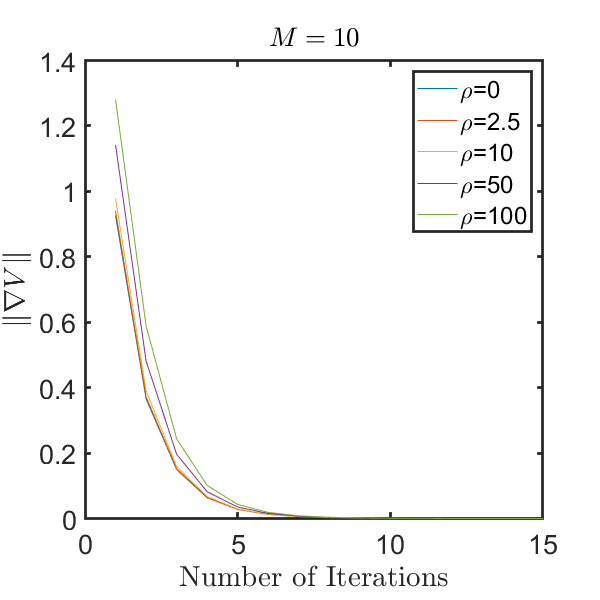"}
	\includegraphics[width=0.3\linewidth]{"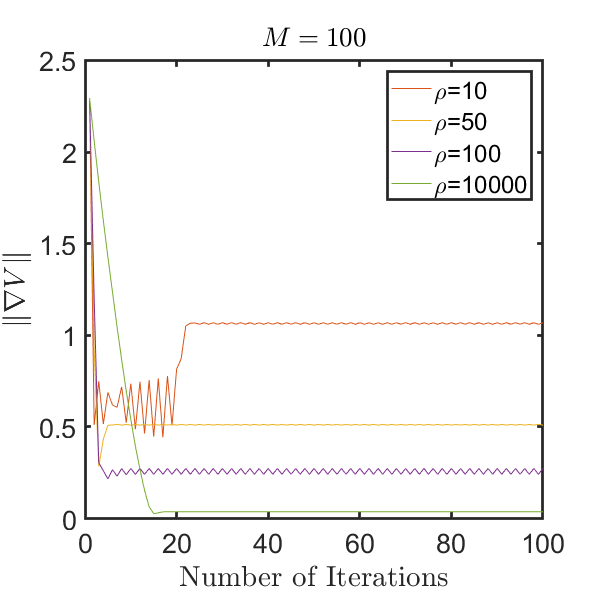"}
	\caption{ Decay of errors (measured by $\|\nabla V(x_k)\|$) in the   iterative proximal minimization   scheme, where $x$-axis is the number of iterations $k$.  $M$ is the steps of gradient descent in the subproblem for  $\widetilde{W}_\rho$ (see Algorithm \ref{alg:algorithm1}).}
	\label{fig:convergence_rate}
\end{figure}

\begin{figure}
	\centering
	\includegraphics[width=0.23\linewidth]{"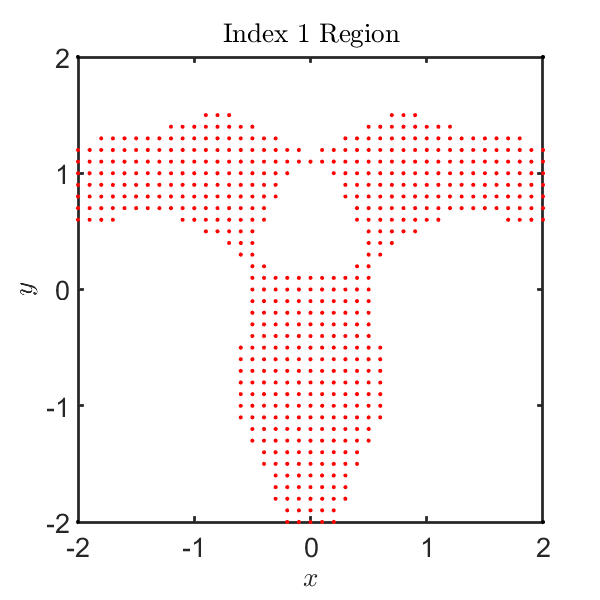"}
	\includegraphics[width=0.23\linewidth]{"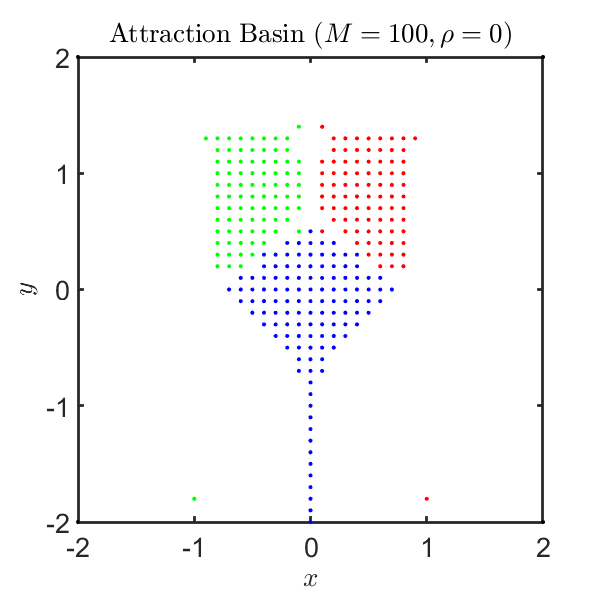"}
	\includegraphics[width=0.23\linewidth]{"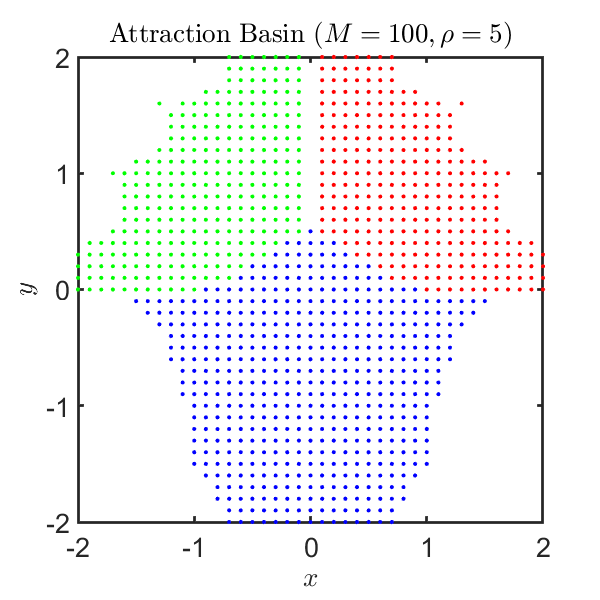"}
	\includegraphics[width=0.23\linewidth]{"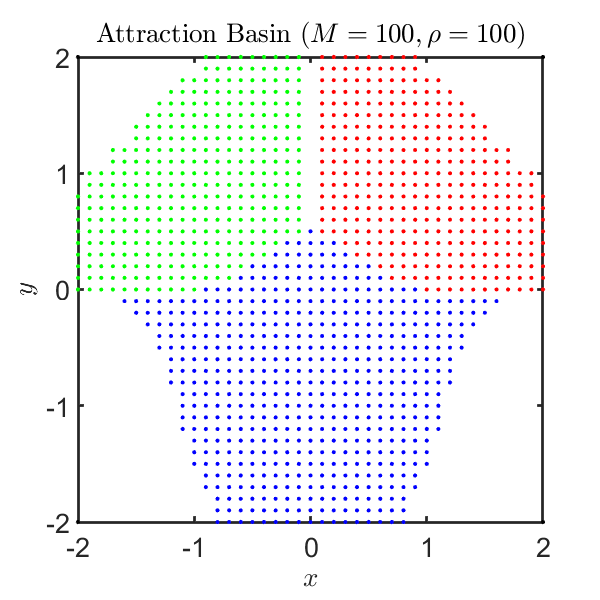"}
	\caption{Comparison of attraction basins towards each of three saddle points when $\rho$ varies. The index-1 region $\Omega_1$ of the function $V$ is also shown.}
	\label{fig:attraction_basin}
\end{figure}

We test different combinations of $\rho$ and $M$ on this example for the IPM algorithm. 
Figure \ref{fig:convergence_rate}
shows  how the errors decay. 
For small $M$, 
the minimization subproblem  in each iteration is solved very inexactly
and the penalty function $d(x,y)$
can barely take effects. So, the
difference in the value of $\rho$ leads to little difference in results.
  If $M$ is   large,   the  effect of  $\rho$
   becomes very important to maintain the convergence. 
 In the last panel of Figure \ref{fig:convergence_rate}, we deliberately select a bad initial point to  show the interesting effect that  the increasing 
 penalty factor 
can improve the convergence by suppressing the oscillations.

We further test  the convergence behaviours by 
examining the basin of attraction of our IPM scheme at $M=100$.
Figure \ref{fig:attraction_basin} exhibits the attraction basin, which is the collection of initial points that  our scheme converges to one of three saddle points under different values of $\rho$.
We can see that the attraction basin significantly expands at  $\rho=5$.
When $\rho=100$, the basin of attraction  can   cover the whole index-1 region $\omega_1$. $x\in\Omega_1$ is a sufficient condition for the minimization problem in each step of iteration to be globally convex as stated in Theorem \ref{th1}, from which we inferred that $x_0\in\Omega_1$ is a sufficient condition for the  iterative minimization  proximal minimization scheme to converge to a saddle point. And this aligns with the observations in this numerical experiment.

 \subsection{Cahn-Hilliard equation}
 The second example is   Cahn-Hilliard equation, which has  been widely used in many complicated moving interface problems in material sciences and fluid dynamics through a phase-field approach \cite{ ShenYang, Fife}. 
  Consider the Ginzburg-Landau free energy on a one dimensional interval $[0,1]$ 
  \begin{equation}
 \label{eqn:F_GL0}
F(\phi) = \int_\Omega \left[ \frac{\kappa^2}{2}
|\nabla \phi(x)|^2+ f(\phi(x)) \right]\,dx, \quad f(\phi) = (\phi^2-1)^2/4,
\end{equation}
with $\kappa=0.04$ and the constant mass $\int \phi dx=0.6$.
 The 
 Cahn-Hilliard (CH) equation \cite{CH-EQ}
 is the $H^{-1}$-gradient flow of $F(\phi)$,
\begin{equation}\label{CH_eq}
\frac{\partial\phi}{\partial t} = \Delta\frac{\delta F}{\delta\phi}=-\kappa^2 \Delta^2 \phi + \Delta (\phi^3-\phi).
\end{equation}
Here $\frac{\delta F}{\delta \phi}$ is the first order variation of $F$
in the standard $L^2$ sense. 

We are interested in the transition state of the Cahn-Hilliard equation, which is the index-1 saddle point of Ginzburg-Landau free energy in $H^{-1}$ Riemannian metric. However, in the calculation  by the original IMF\cite{IMF2015,convex_IMF}, the convergence effect relies on a good initial state as well as  the inner iteration number $M$.
In the IPM method, we take $d(x,y)=|x-y|^4$, the auxiliary functional then becomes
\begin{align}
    \widetilde{W}_\rho(\phi,\phi^{(k)}) = \int_\Omega \Big[ \frac{\kappa^2}{2}|\nabla\phi|^2 + f(\phi) - \kappa^2 |\nabla\hat\phi|^2 - 2f(\hat\phi) \Big]\,dx + \rho\int_\Omega |\phi-\phi^{(k)}|^4 \,dx.
\end{align}
The gradient flow of $\widetilde{W}_\rho(\phi,\phi^{(k)})$ in $H^{-1}$ metric is
$   \frac{\partial\phi}{\partial t} = \Delta \delta_\phi \widetilde{W}_\rho(\phi),
$
with $
 \delta_\phi \widetilde{W}_\rho(\phi) = -\kappa^2 \Delta \phi + (\phi^3-\phi) + 2 \inpd{\bv v_1}{ \kappa^2 \Delta \hat\phi - (\hat\phi^3 + \hat\phi) }_{L^2} \bv v_1 + 4\rho(\phi-\phi^{(k)})^3, $ and $\hat\phi=\phi^{(k)}+\inpd{\bv v_1}{\phi-\phi^{(k)}}_{H^{-1}}\bv v_1$.

\begin{figure}[htbp]
\centering
\subfloat[Transition state]{\label{fig:saddle}\includegraphics[scale=0.40]{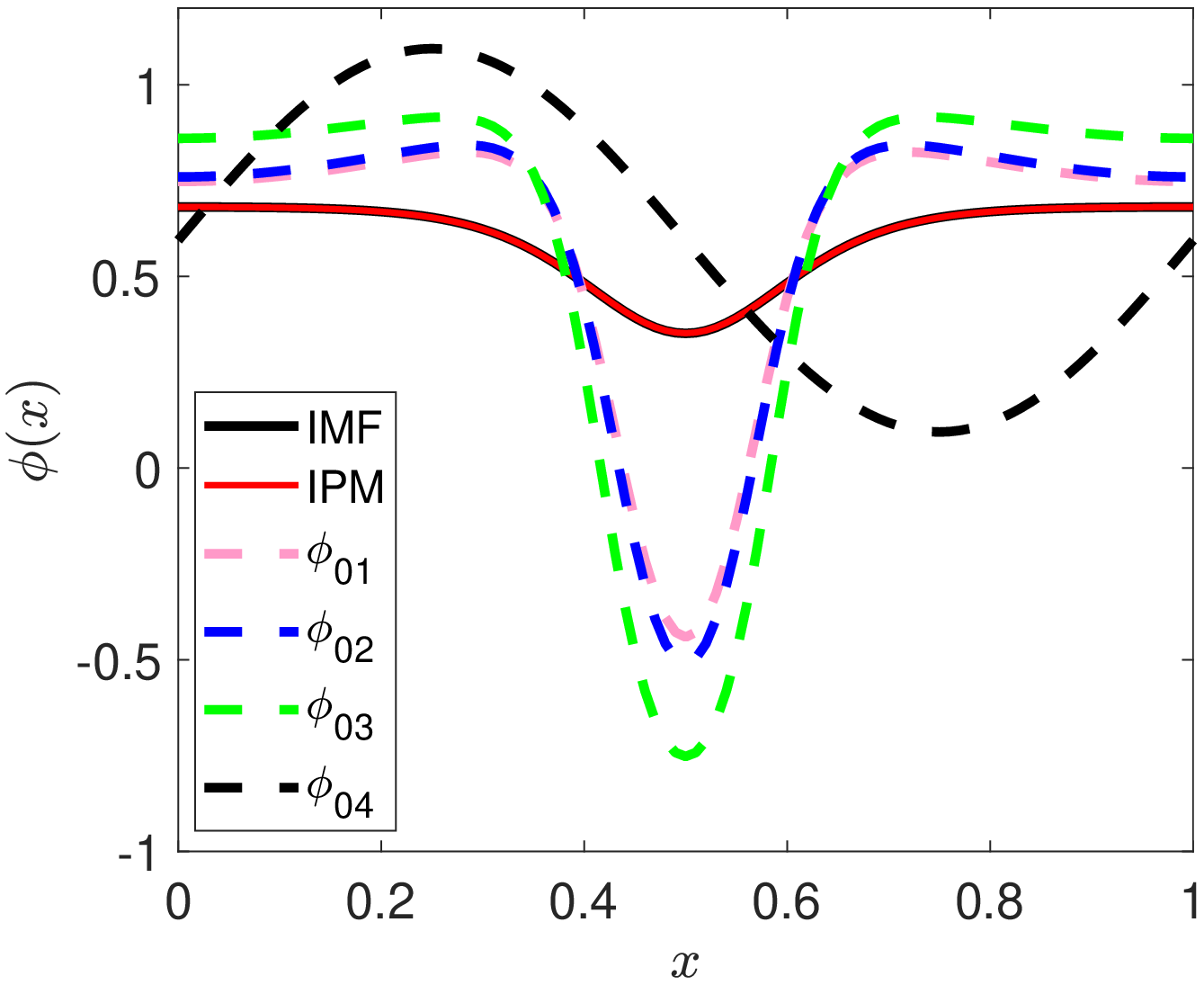}}
\subfloat[Decay of error]{\label{fig:Modif_IMF_convergence_rate}\includegraphics[scale=0.40]{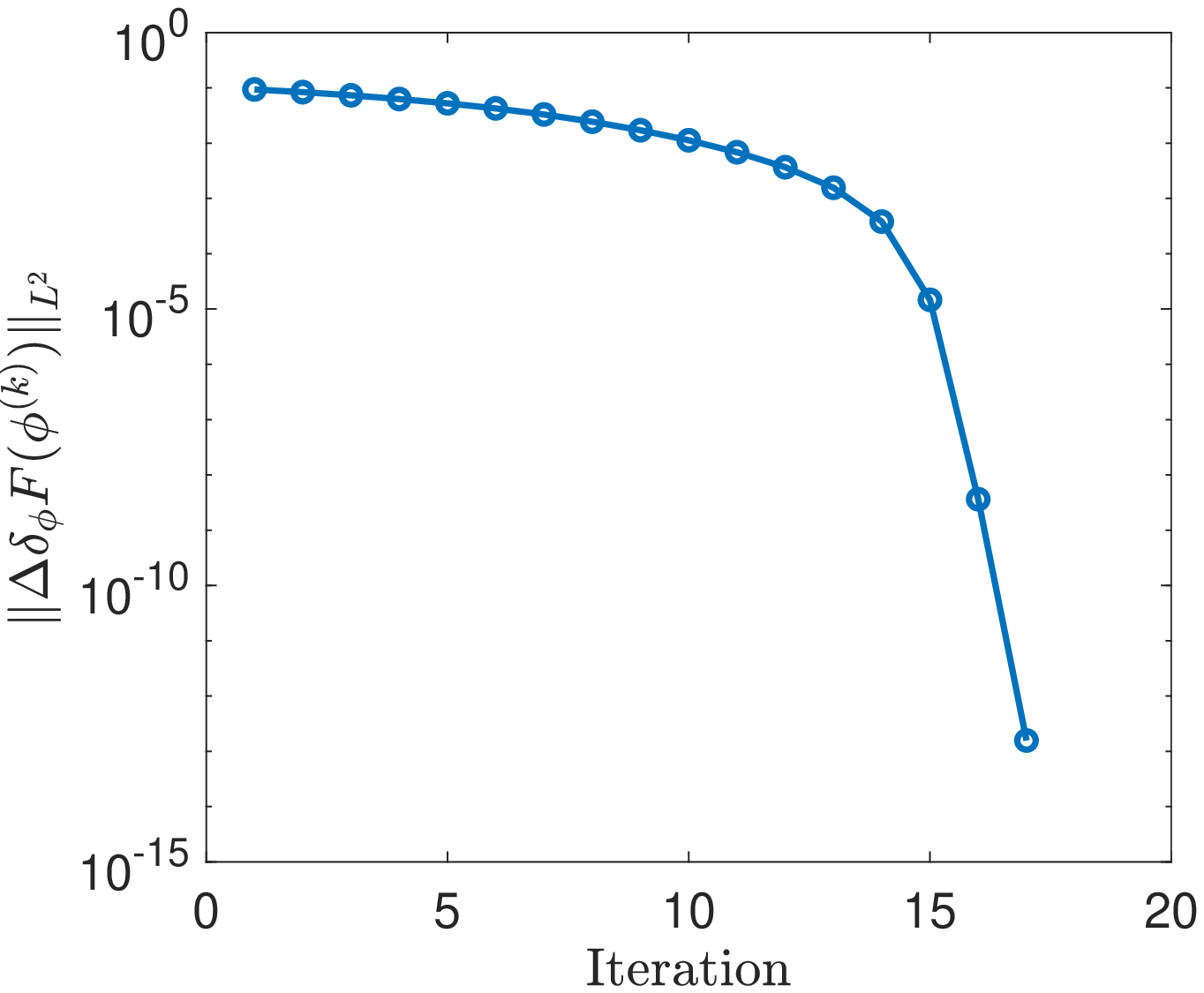}}
\caption{(a): Transition state (solid curves) computed by the IMF and IPM from different initial states.  The pink, blue and green dashed lines are the initial states correspond to the three states in Table \ref{1D_case} which are taken from the minimum energy path, while the dark dashed line is from the initial   $\phi_{04}=\sin(2\pi x)$.
 (b): The decay of the error $\| \Delta \delta_\phi F(\phi^{(k)})\|_{L^2}$ measured by the $L^2$ norm of the $H^{-1}$-gradient. }
 \label{saddle}
\end{figure}

In the numerical simulation, we take $\rho=100$ for the penalty factor and use the uniform
  mesh grid  for spatial discretization  $\{x_i = i h , i=0, 1, 2,\ldots, N\}. ~h = 1/N.$ $N=100$, $\Delta t = 0.1$. The periodic boundary condition is considered.
The saddle point of $F(\phi)$ is reproduced (see Figure \ref{fig:saddle}), which is the same as the result in the references\cite{convex_IMF, ProjIMF}.
Besides, the quadratic convergence rate of the IPM algorithm is  also verified empirically; see Figure \ref{fig:Modif_IMF_convergence_rate}. 
  In order to illustrate the advantage of this method, we make comparison of the convergence results between the original IMF ($\rho=0$) and the proximal method ($\rho=100$) here, starting from different initial states and  with the different  inner iteration number $M$.  
     Table \ref{1D_case} shows the convergence/divergence results for three initial states $\phi_{01}, \phi_{02} $  and $\phi_{03}$. The convergence/divergence result for the initial $\phi_{04}$ is the same as the result for $\phi_{03}$. We find that the farther the initial state is  away from the saddle point, the smaller number of inner iterations    the original IMF can tolerate, but  the IPM can ensure convergence regardless of all initial states and inner iteration numbers tested here.

\begin{table}[htbp]
{
\footnotesize   \caption{Comparison of numerical convergence from three differential initial guesses $\phi_{01}$,$\phi_{02}$ and $\phi_{03}$ shown in Figure \ref{fig:saddle}. ``IMF" means the original IMF; `` IPM" is the new method of iterative proximal minimization in this paper. 
$M$ is the number of gradient descent steps in  minimizing  the auxiliary functions. 
``$\checkmark$" and ``\ding{55}" mean the convergent and divergent results, respectively. }
\label{1D_case}
\begin{center}
    \begin{tabular}{|c|r|r|r|r|r|r|}
        \hline
        \multirow{2}{*}{$M$} & \multicolumn{2}{c|}{$\phi_{01}$} & \multicolumn{2}{c|}{$\phi_{02}$} & \multicolumn{2}{c|}{$\phi_{03}$} \\ \cline{2-7}
         &  IMF  & IPM  &  IMF  & IPM &  IMF  & IPM  \\ \hline
         10 & $\checkmark$ & $\checkmark$ & $\checkmark$ & $\checkmark$ & $\checkmark$ & $\checkmark$ \\ \hline 
          100 & $\checkmark$ & $\checkmark$ & $\checkmark$ & $\checkmark$ & \ding{55} & $\checkmark$ \\ \hline 
           200 & $\checkmark$ & $\checkmark$ & \ding{55} & $\checkmark$ & \ding{55} & $\checkmark$ \\ \hline 
            500 & \ding{55} & $\checkmark$ & \ding{55} & $\checkmark$ & \ding{55} & $\checkmark$ \\ \hline 
    \end{tabular}
  \end{center}
 }
\end{table}

  \section{Conclusion} \label{con}
  The calculation  of relevant index-1 saddle points to transitions on a potential energy surface 
  is an important computational task for rare event and phase transitions in chemistry and material science. 
We have established the equivalent connection between the index-1 saddle point of any  function $V$ with continuous Hessian and the Nash equilibrium of a differential game constructed based on the iterative minimization formulation \cite{IMF2015} for saddle points.
The numerical contribution is  a new iterative minimization algorithm  with the proximal penalty function to enhance the robustness. The generalization to any Morse index $k$ is also discussed.   The saddle-point calculation is in general still a formidable challenge compared to the  
gradient descent method for minimum points. It might be rewarding for a  further exploration of  the existing saddle-point search methods based on the minimal mode  and the algorithmic game theory.





 \bibliography{MsGAD,gad,CVXIMF,my,ms}
 
\bibliographystyle{siam} 

\end{document}